\documentclass{amsart}

\DeclareFontFamily{U}{mathb}{\hyphenchar\font45}
\DeclareFontShape{U}{mathb}{m}{n}{
      <5> <6> <7> <8> <9> <10> gen * mathb
      <10.95> mathb10 <12> <14.4> <17.28> <20.74> <24.88> mathb12
      }{}
\DeclareSymbolFont{mathb}{U}{mathb}{m}{n}
\DeclareMathSymbol{\righttoleftarrow}{3}{mathb}{"FD}

\usepackage[english]{babel}
\usepackage[utf8]{inputenc}
\usepackage{graphicx}
\usepackage{xcolor}
\usepackage{longtable}
\usepackage{float}

\usepackage[all]{xy}

\usepackage{amsmath}
\usepackage{amsthm}
\usepackage{amssymb}
\usepackage{amscd}
\usepackage{tikz}
\usepackage{csquotes}
\usepackage{imakeidx}
\usepackage{appendix}
\usepackage{tikz-cd}
\usepackage{verbatim}
\usepackage{enumitem}
\usepackage{multicol}
\usepackage{aliascnt}
\usepackage[backref=page]{hyperref}
\usepackage{cleveref}
\usepackage{mathtools}

\hypersetup{colorlinks=true,linkcolor=blue,anchorcolor=blue,citecolor=blue}

\newtheorem{thm}{Theorem}[section]
\crefname{thm}{Theorem}{Theorems}

\newaliascnt{prop}{thm}
\newtheorem{prop}[prop]{Proposition}
\aliascntresetthe{prop}
\crefname{prop}{Proposition}{Propositions}

\newaliascnt{lemma}{thm}
\newtheorem{lemma}[lemma]{Lemma}
\aliascntresetthe{lemma}
\crefname{lemma}{Lemma}{Lemmas}

\newaliascnt{cor}{thm}
\newtheorem{cor}[cor]{Corollary}
\aliascntresetthe{cor}
\crefname{cor}{Corollary}{Corollaries}

\newaliascnt{conj}{thm}
\newtheorem{conj}[conj]{Conjecture}
\aliascntresetthe{conj}
\crefname{conj}{Conjecture}{Conjectures}

\newaliascnt{question}{thm}
\newtheorem{question}[question]{Question}
\aliascntresetthe{question}
\crefname{question}{Question}{Questions}

\newaliascnt{situation}{thm}

\aliascntresetthe{situation}
\crefname{situation}{Situation}{Situations}

\theoremstyle{definition}
\newaliascnt{defin}{thm}
\newtheorem{defin}[defin]{Definition}
\aliascntresetthe{defin}
\crefname{defin}{Definition}{Definitions}

\newaliascnt{construction}{thm}

\aliascntresetthe{construction}
\crefname{construction}{Construction}{Constructions}
\Crefname{construction}{Construction}{Constructions}

\newaliascnt{claim}{thm}

\aliascntresetthe{claim}
\crefname{claim}{Claim}{Claims}
\Crefname{claim}{Claim}{Claims}

\newtheorem{example}[thm]{Example}
\newtheorem{exercise}[thm]{Exercise}

\theoremstyle{remark}
\newaliascnt{rmk}{thm}
\newtheorem{rmk}[rmk]{Remark}
\aliascntresetthe{rmk}
\crefname{rmk}{Remark}{Remarks}
\Crefname{rmk}{Remark}{Remarks}

\numberwithin{equation}{section}

\newcommand{\F}{\mathbb F}
\newcommand{\C}{\mathbb C}

\renewcommand{\P}{\mathbb P}
\newcommand{\Spec}{\operatorname{Spec}}

\renewcommand{\phi}{\varphi}
\newcommand{\ed}{\operatorname{ed}}

\newcommand{\ang}[1]{\langle{#1}\rangle}

\newcommand{\eqto}{\stackrel{\lower1.5pt\hbox{$\scriptstyle\sim\,$}}\to}
\newcommand{\eqdashto}{\stackrel{\lower1.5pt\hbox{$\scriptstyle\sim\,$}}\dashrightarrow}

\makeatletter
\@namedef{subjclassname@2020}{\textup{2020} Mathematics Subject Classification}
\makeatother

\definecolor{MyDarkGreen}{rgb}{0.0,0.5,0.0}

\DeclareFontEncoding{OT2}{}{}
\DeclareTextFontCommand{\textcyr}{\fontencoding{OT2}\fontfamily{wncyr}\fontseries{m}\fontshape{n}\selectfont}

\newcommand{\bthe}{\begin{thm}}
\newcommand{\ethe}{\end{thm}}
\newcommand{\ble}{\begin{lemma}}
\newcommand{\ele}{\end{lemma}}
\newcommand{\bpr}{\begin{prop}}
\newcommand{\epr}{\end{prop}}
\newcommand{\bco}{\begin{cor}}
\newcommand{\eco}{\end{cor}}
\newcommand{\bde}{\begin{defin}}
\newcommand{\ede}{\end{defin}}
\newcommand{\brem}{\begin{rmk}}
\newcommand{\erem}{\end{rmk}}
\newcommand{\bexe}{\begin{exercise}}
\newcommand{\eexe}{\end{exercise}}
\newcommand{\bexa}{\begin{example}}
\newcommand{\eexa}{\end{example}}
\newcommand{\bconj}{\begin{conj}}
\newcommand{\econj}{\end{conj}}
\newcommand{\bques}{\begin{question}}
\newcommand{\eques}{\end{question}}

\title[A counterexample to Duncan's conjecture]{A counterexample to a conjecture of Duncan on versal actions}

\author{Federico Scavia}
\address{CNRS\\
Institut Galil\'ee\\
Universit\'e Sorbonne Paris Nord\\
99 avenue Jean-Baptiste Cl\'ement, 93430\\
Villetaneuse, France}
\email{scavia@math.univ-paris13.fr}


\date{July 27, 2026}

\begin{document}

\pagestyle{headings}

\begin{abstract}
We disprove a conjecture of Duncan asserting that versality of a faithful
finite group action can be detected on Sylow subgroups. Our counterexample
is an action of $(C_7\rtimes C_3)\times C_2$ on a degree-$2$ Del Pezzo surface.
\end{abstract}

\vspace*{-0.5cm}
\maketitle

\section{Introduction}

Let $k$ be a field, and let $G$ be a finite group. By definition, a $k$-variety is a separated $k$-scheme of finite type, and a $G$-variety is a $k$-variety together with a left $G$-action over $k$. Among $G$-varieties, versal $G$-varieties (see \S\ref{sec:prelim} for the definition) play a special role. For example, the essential dimension $\ed_k(G)$ of $G$ over $k$ may be defined as the smallest dimension of a versal faithful $G$-variety over $k$. The following conjecture was formulated by Duncan \cite{Duncan-ed2}. 

\begin{conj}[Duncan]\label{conj:duncan}
Let $k$ be a field, and let $G$ be a finite group acting faithfully on a $k$-variety $X$. For every prime $p$, let $G_p$ be a Sylow $p$-subgroup of $G$. Then $X$ is $G$-versal
if and only if $X$ is $G_p$-versal for every prime $p$.
\end{conj}

Duncan formulated \Cref{conj:duncan} while classifying finite
groups of essential dimension $2$. For a smooth complete toric variety, he
used Cox rings and universal torsors to prove that versality is detected on
Sylow subgroups \cite[Corollary~3.6]{Duncan-ed2}. Since versality is a
birational invariant, this led him to \Cref{conj:duncan}.

Duncan and Reichstein placed \Cref{conj:duncan} in a broader
framework relating versality to rational points on twisted varieties
\cite{Duncan-Reichstein}. They characterized weak versality, versality, very versality, and $p$-versality in terms of rational points on all twists of $X$. Their results illustrate the subtlety of \Cref{conj:duncan}: indeed, the weaker hypothesis of $p$-versality for every prime $p$ does not imply versality in general, even for the trivial group over a non-algebraically closed field $k$. The reason is that a $k$-variety with a zero-cycle of degree one need not have a Zariski-dense set of $k$-points; it need not even have a $k$-point.

Several results in \cite{Duncan-Reichstein} may be viewed as special cases or
close analogues of \Cref{conj:duncan}.  For projective space, the corresponding twists are Severi--Brauer varieties, and all notions of versality are equivalent to the existence of a lift of the projective action to a linear representation; see \cite[Proposition~9.1]{Duncan-Reichstein}.  For quadrics,
Springer's theorem yields a positive criterion involving a Sylow $2$-subgroup
\cite[Theorem~10.2]{Duncan-Reichstein}.  For smooth cubic hypersurfaces, the
analogous statement is related to the Cassels--Swinnerton-Dyer conjecture; see \cite[Theorem~10.5]{Duncan-Reichstein}. \Cref{conj:duncan} also enters the discussion of the Klein cubic threefold:
Duncan and Reichstein showed that it would imply
$\ed_{\C}(\operatorname{PSL}_2(\F_{11}))\leq 3$, thus completing the classification of finite simple groups of essential dimension $3$; see \cite[Example 10.7]{Duncan-Reichstein}.

We show that \Cref{conj:duncan} is false already in dimension $2$. In fact, we prove something stronger: there exists a non-versal $G$-action on a rational surface $S$ whose restriction to every Sylow subgroup of $G$ is very versal (see \S\ref{sec:prelim} for the definition).

\begin{thm}\label{thm:main}
Let $\zeta\in\C$ be a primitive seventh root of unity, and let
\[
 S=\{w^2=x^3y+y^3z+z^3x\}\subset\mathbb P(1,1,1,2).
\]
The surface $S$ is a Del Pezzo\footnote{I write \emph{Del Pezzo}, rather than \emph{del Pezzo}, following contemporary official usage. In particular, documents produced during Pasquale Del Pezzo's lifetime and preserved in his personal file at the Historical Archive of the Italian Senate record his name as ``Del Pezzo Pasquale''; see \href{https://patrimonio.archivio.senato.it/inventario/scheda/ufficio-segreteria/IT-SEN-095-003889/del-pezzo-pasquale}{Archivio storico del Senato della Repubblica, \emph{Del Pezzo Pasquale}}. I thank Matteo Tamiozzo for bringing this usage to my attention and for providing the link.}
 surface of degree $2$. Let
\[G\coloneqq \langle \rho,\sigma,\tau : \rho^7=\sigma^3=\tau^2=[\tau,\rho]=[\tau,\sigma]=1, \sigma\rho\sigma^{-1}=\rho^4\rangle\simeq (C_7\rtimes C_3)\times C_2.\]
The formulas
\begin{align*}
 \rho[x:y:z:w]&=[\zeta x:\zeta^4y:\zeta^2z:w],\\
 \sigma[x:y:z:w]&=[y:z:x:w],\\
 \tau[x:y:z:w]&=[x:y:z:-w]
\end{align*}
define a faithful $G$-action on $S$. 

\begin{enumerate}
    \item The restriction of the action to every Sylow
subgroup of $G$ is very versal.
\item We have $\ed_{\C}(G)=3$, and hence the action of $G$ on $S$ is not
versal.
\end{enumerate}
\end{thm}

For each Sylow subgroup $G_p$ of $G$, we construct a nonconstant $G_p$-equivariant morphism $\mathbb P^1_\C\to S$. Using a theorem of Salgado--Testa--V\'arilly-Alvarado \cite[Theorem~17]{Salgado-Testa-VA}, we show that every twist of $S$ by a $G_p$-torsor is unirational. The criterion of Duncan--Reichstein (\Cref{prop:twist-criterion}) then implies that the $G_p$-action on $S$ is very versal for every prime $p$. On the other hand, it is not difficult to see that $\ed_{\C}(G)=3$; in particular, the $G$-action on the surface $S$ is not versal.

The surface $S$ and the action of $G$ already appeared in
\cite[Example~1.9]{Duncan-delPezzo}. Duncan observed that every abelian
subgroup of $G$ has a fixed point on $S$ and that $\ed_{\C}(G)>2$. The novel part of \Cref{thm:main} is that the restriction of the action to every Sylow subgroup of $G$ is very versal, and hence versal.

\section{Preliminaries on versality}\label{sec:prelim}

Let $k$ be a field, let $H$ be a finite discrete group, and let $X$ be an
integral $H$-variety. Following \cite{Duncan-Reichstein}, we say
that $X$ is
\begin{enumerate}[label=\textup{(\roman*)}]
 \item \emph{weakly versal} if, for every infinite field extension $K/k$
 and every $H$-torsor $T$ over $K$, there exists an $H$-equivariant
 $k$-morphism $T\to X$;
 \item \emph{versal} if every dense $H$-invariant open subvariety of $X$ is
 weakly versal;
 \item \emph{very versal} if there exist a finite-dimensional $k$-linear
 $H$-representation $V$ and a dominant $H$-equivariant rational map
 $V\dashrightarrow X$.
\end{enumerate}

Very versality implies versality, and versality implies weak versality; see
\cite[Equation~(1.1)]{Duncan-Reichstein}.

Let $K/k$ be an infinite field extension and let $T\to\Spec K$ be a right
$H$-torsor. The pair $(T,K)$ is called an $H$-\emph{twisting pair}. If $X$
is an $H$-variety, its twist is
\[
 {}^T\!X=(T\times_k X)/H,
\]
where $H$ acts on $T\times_kX$ on the right by
\[
 (t,x)\cdot h=(th,h^{-1}x).
\]
The variety ${}^T\!X$ is a $K$-form of $X_K$. Twisting by a fixed $H$-twisting pair is functorial with respect to $H$-equivariant morphisms and $H$-equivariant rational maps.

\begin{prop}[Duncan--Reichstein]\label{prop:twist-criterion}
Let $k$ be a field, let $H$ be a finite discrete group, and let $X$ be an integral quasi-projective $H$-variety. Then:
\begin{enumerate}[label=\textup{(\roman*)}]
 \item $X$ is weakly versal if and only if
 ${}^T\!X(K)\ne\varnothing$ for every $H$-twisting pair $(T,K)$;
 \item $X$ is versal if and only if ${}^T\!X(K)$ is Zariski dense in
 ${}^T\!X$ for every $H$-twisting pair $(T,K)$;
 \item $X$ is very versal if and only if ${}^T\!X$ is $K$-unirational
 for every $H$-twisting pair $(T,K)$.
\end{enumerate}
\end{prop}

\begin{proof}
See \cite[Theorem~1.1(a)--(c)]{Duncan-Reichstein}.
\end{proof}

The following criterion for very versality of group actions on degree-$2$
Del Pezzo surfaces, based on a theorem of Salgado--Testa--V\'arilly-Alvarado \cite[Theorem~17]{Salgado-Testa-VA}, will be used in the proof of \Cref{thm:main}.

\begin{lemma}\label{lem:fixed-rational-curve}
Let $H$ be a cyclic group, let $X$ be a degree-$2$ Del Pezzo surface over $\C$ with an action of $H$. Suppose that there is a nonconstant $H$-equivariant morphism $\nu:\mathbb P^1_\C\to X$. Then $X$ is $H$-very versal.
\end{lemma}

\begin{proof}
Let $h$ be a generator of $H$. Choose a lift of the image of $h$ in
$\operatorname{PGL}_2(\C)$ to an element of $\operatorname{GL}_2(\C)$. Since
$\C$ is algebraically closed, this lift has an eigenvector. The corresponding
$\C$-point of $\mathbb P^1_\C$ is fixed by $h$, and hence by $H$.

Let $(T,K)$ be an $H$-twisting pair. Twisting $\nu$ gives a morphism
\[
 {}^T\!\nu:{}^T\!(\mathbb P^1_\C)\longrightarrow{}^T\!X.
\]
The $H$-fixed point of $\mathbb P^1_\C$ descends to a $K$-point of
${}^T\!(\mathbb P^1_\C)$. Since ${}^T\!(\mathbb P^1_\C)$ is a smooth projective
geometrically integral curve of genus zero over $K$, it is isomorphic to $\P^1_K$. The twist ${}^T\!X$ is again a degree-$2$ Del Pezzo surface, and the morphism ${}^T\!\nu$ is nonconstant, since it becomes isomorphic to $\nu$ after extending scalars to a splitting field of $T$. By a theorem of Salgado--Testa--V\'arilly-Alvarado
\cite[Theorem~17]{Salgado-Testa-VA}, a degree-$2$ Del Pezzo surface in
characteristic zero admitting a nonconstant morphism from $\mathbb P^1$ is
unirational. Thus ${}^T\!X$ is $K$-unirational for every $H$-twisting pair,
and \Cref{prop:twist-criterion}(iii) shows that $X$ is $H$-very versal.
\end{proof}

\section{Proof of Theorem~\ref{thm:main}}\label{sec:proof}

We now prove \Cref{thm:main}. Let $x,y,z$ be homogeneous coordinates on $\P^2_\C$, and let $x,y,z,w$ be weighted homogeneous coordinates on $\P(1,1,1,2)$. Set
\[
 f(x,y,z)\coloneqq x^3y+y^3z+z^3x,
 \qquad
 B\coloneqq \{f(x,y,z)=0\}\subset\mathbb P^2_\C,
\]
and let $\pi:S\to\mathbb P^2_\C$ be the double cover ramified over $B$:
\[
 S\coloneqq\{w^2=f(x,y,z)\}\subset\mathbb P(1,1,1,2).
\]
The curve $B$ is smooth, and hence so is the surface $S$.

Write $H$ for the class of a line in $\mathbb P^2_\C$.  Since $B\in|4H|$,
the canonical bundle formula for a smooth double cover gives
\[
 K_S=\pi^*(K_{\P^2_\C}+2H)=\pi^*(-H).
\]
Consequently $-K_S=\pi^*H$ is ample and 
\[
 K_S^2=(\pi^*H)^2=2H^2=2.
\] Thus $S$ is a Del Pezzo surface of degree $2$.

Define
\[G\coloneqq \langle \rho,\sigma,\tau : \rho^7=\sigma^3=\tau^2=[\tau,\rho]=[\tau,\sigma]=1, \sigma\rho\sigma^{-1}=\rho^4\rangle\simeq (C_7\rtimes C_3)\times C_2,\]
and set
\[
 G_7\coloneqq\ang{\rho},\qquad G_3\coloneqq\ang{\sigma},\qquad G_2\coloneqq\ang{\tau},\qquad G_p\coloneqq\{1\}\,\,(p\notin\{2,3,7\}),
\]
so that $G_p$ is a cyclic Sylow $p$-subgroup of $G$ for each prime $p$.

Fix a primitive seventh root of unity $\zeta\in\C$.  The formulas
\begin{align*}
 \rho[x:y:z:w]&=[\zeta x:\zeta^4y:\zeta^2z:w],\\
 \sigma[x:y:z:w]&=[y:z:x:w],\\
 \tau[x:y:z:w]&=[x:y:z:-w]
\end{align*}
define a faithful $G$-action on $S$.

\begin{prop}\label{prop:sylow-very-versal}
The surface $S$ is $G_p$-very versal for every prime $p$.
\end{prop}

\begin{proof}
If $p\notin\{2,3,7\}$, then $G_p$ is trivial and the conclusion follows from the fact that $S$ is rational and \Cref{prop:twist-criterion}(iii).

Suppose that $p\in\{2,7\}$. The inverse image of the tangent line
$y=0$ to $B$ at $[1:0:0]$ is the cuspidal rational curve
\[
 C\coloneqq\{y=0,\ w^2=z^3x\}\subset S.
\]
Its normalization map is
\[
 \nu:\mathbb P^1_\C\longrightarrow C\subset S,
 \qquad
 [u:v]\longmapsto[u^2:0:v^2:uv^3].
\]
Define actions of $G_2$ and $G_7$ on $\P^1_\C$ by the formulas
\[
 \tau[u:v]=[u:-v],
 \qquad
 \rho[u:v]=[u:\zeta^4v].
\]
The morphism $\nu$ is equivariant for these actions:
\[
 \nu(\tau[u:v])
 =[u^2:0:v^2:-uv^3]
 =\tau(\nu[u:v]),
\]
\[
 \nu(\rho[u:v])
 =[u^2:0:\zeta v^2:\zeta^5uv^3]
 =[\zeta u^2:0:\zeta^2v^2:uv^3]
 =\rho(\nu[u:v]).
\]
It follows from
\Cref{lem:fixed-rational-curve} that $S$ is $G_2$-very versal and
$G_7$-very versal.

It remains to consider the case $p=3$. Let
\[
 L\coloneqq\{x+y+z=0\}\subset\mathbb P^2_\C.
\]
Substituting $z=-x-y$, we obtain $f(x,y,-x-y)=-(x^2+xy+y^2)^2$. Thus the inverse image of $L$ in $S$ is the union of the two curves
\[
 E_\pm=
 \{x+y+z=0,\quad w=\pm i(x^2+xy+y^2)\}.
\]
Each of these curves is isomorphic to $\mathbb P^1_\C$. Indeed, the
restriction of $\pi$ induces an isomorphism $E_\pm\to L$, whose inverse is
\[
 [x:y:z]\longmapsto
 [x:y:z:\pm i(x^2+xy+y^2)].
\]
Moreover, since $\sigma(L)=L$, the two irreducible components $E_+$ and $E_-$ of $\pi^{-1}(L)$ are permuted by $\sigma$. Since $\sigma$ has order $3$, its induced permutation of the two-element set $\{E_+,E_-\}$ is trivial, that is, $\sigma$ preserves each of $E_+$ and
$E_-$. Consequently, identifying $E_+$ with $\mathbb P^1_\C$, the inclusion $E_+\hookrightarrow S$ yields a nonconstant $G_3$-equivariant morphism $\mathbb P^1_\C\to S$. By \Cref{lem:fixed-rational-curve}, we conclude that $S$ is $G_3$-very versal.
\end{proof}

\begin{prop}\label{prop:ed-nonversal}
We have $\ed_{\C}(G)=3$. In particular, the $G$-variety $S$ is not versal.
\end{prop}

\begin{proof}
We first prove that $\ed_{\C}(G)>2$. Since $G$ is not trivial, cyclic, or
dihedral of order $2n$ with $n$ odd, we have $\ed_{\C}(G)>1$; see
\cite[Proposition~2.4(c) and (d)]{Duncan-ed2}. Suppose that
$\ed_{\C}(G)=2$. Since the center of $G$ is nontrivial,
\cite[Proposition~2.5]{Duncan-ed2} implies that $G$ embeds into
$\operatorname{GL}_2(\C)$. In particular, the subgroup
\[
 H\coloneqq\ang{\rho,\sigma}\simeq C_7\rtimes C_3
\]
has a faithful two-dimensional complex representation $V$.

Since $\rho$ has finite order, its action on $V$ is diagonalizable. Let $v\in V$ be a
$\rho$-eigenvector, and let $\lambda\in\mu_7(\C)$ be its eigenvalue. From the relations $\rho^7=\sigma^3=1$ and $\sigma\rho\sigma^{-1}=\rho^4$, we obtain
\[
 \rho(\sigma v)
 =\sigma\rho^2v
 =\lambda^2\sigma v,\qquad \rho(\sigma^2v)
 =\sigma^2\rho^4v
 =\lambda^4\sigma^2v.
\]
Thus $v$, $\sigma v$, and $\sigma^2v$ are $\rho$-eigenvectors with
eigenvalues $\lambda$, $\lambda^2$, and $\lambda^4$, respectively. If
$\lambda\ne1$, these three eigenvalues are distinct, which is impossible
because $V$ has dimension $2$. Therefore every eigenvalue of $\rho$ is equal to $1$.
Since $\rho$ is diagonalizable, we deduce that it acts trivially on $V$, contradicting the
faithfulness of the representation. Hence $\ed_{\C}(G)>2$.

For the reverse inequality, the formulas
\[
\varphi(\rho)=
\begin{pmatrix}
\zeta & 0 & 0\\
0 & \zeta^4 & 0\\
0 & 0 & \zeta^2
\end{pmatrix},
\qquad
\varphi(\sigma)=
\begin{pmatrix}
0 & 1 & 0\\
0 & 0 & 1\\
1 & 0 & 0
\end{pmatrix},
\qquad
\varphi(\tau)=
\begin{pmatrix}
-1 & 0 & 0\\
0 & -1 & 0\\
0 & 0 & -1
\end{pmatrix}
\]
define a faithful three-dimensional complex representation of $G$.
Thus $\ed_{\C}(G)\leq3$. We conclude that $\ed_{\C}(G)=3$. The fact that $S$ is not $G$-versal now follows from the definition of essential dimension.
\end{proof}

\begin{proof}[Proof of \Cref{thm:main}]
The conclusion follows from \Cref{prop:sylow-very-versal} and
\Cref{prop:ed-nonversal}.
\end{proof}

\end{document}